\documentclass{article}
\usepackage{graphicx}
\usepackage{mathtools}
\usepackage{xcolor}
\usepackage{amssymb}
\usepackage{amsmath}
\usepackage{amssymb,amsmath, calligra,mathrsfs}
\usepackage{tikz}
\usepackage{tikz-cd}
\usepackage{tikz-3dplot}
\usetikzlibrary{shapes,calc,positioning}
\usepackage{tkz-euclide}
\usepackage[caption=false]{subfig}
\usepackage{amssymb,amsmath,amsthm, calligra,mathrsfs}
\usepackage[english]{babel}
\usepackage{hyphenat}
\usepackage{float}

\newtheorem{theorem}{Theorem}[section]
\newtheorem{corollary}[theorem]{Corollary}
\newtheorem{lemma}[theorem]{Lemma}
\newtheorem{proposition}[theorem]{Proposition}

\newtheorem*{theorem*}{Theorem}
\theoremstyle{definition}
\newtheorem{definition}[theorem]{Definition}
\newtheorem{example}[theorem]{Example}
\newtheorem{remark}[theorem]{Remark}

\begin{document}
\title{Toric non-Gorenstein loci}
%

%
\author{Kemal Rose}
%
%


%
\maketitle              
\begin{abstract}
We describe the non-Gorenstein loci of normal toric varieties.
In the case of Hibi rings, a combinatorial
description`
is provided in terms of the underlying partially ordered set.
As a non-toric application we compute the dimensions of the non-Gorenstein loci of the first
secant variety of Segre varieties.
\end{abstract}

\section{Introduction}

In this paper we study a local measure for singularity which
is based on a natural generalization of smoothness, the Gorenstein property.

The combinatorial properties of a polyhedral cone are deeply linked to the
singularities of the associated toric variety.
There are several results in this direction based on the following
theorem proved in [HHS19]:
the non-Gorenstein locus of a graded ring $R$ with canonical module $\omega$ is the the set of
prime ideals containing the trace ideal
\[
\operatorname{tr}(\omega) = \sum_{\varphi \in \operatorname{Hom}_R(\omega, R) } \varphi(\omega). 
\]
This makes the ideal $\operatorname{tr}(\omega)$ a measure for singularity.

The aim of this paper is to provide a clear relation between the combinatorics 
of the polyhedral cone and the geometry of the toric variety.
In particular, Theorem \ref{thm: charac} in Chapter 2 describes the non-Gorenstein loci of
toric varieties.
This is a vast generalization of results in \cite{Measuring} (Theorem 4.9), where
Herzog, Mohammadi and Page give criteria for certain simplicial toric varieties to be
Gorenstein on the punctured spectrum.

As an application of the toric framework
developed in section 2, we proceed with an investigation of
non-Gorenstein loci of Hibi rings in section 3 and 4, and
as a non-toric application we study secants of Segre varieties in the last section. 
The separate treatment of the toric case enables us to give particularly
conceptual results and comparatively short and clear proofs.

A Hibi ring $k[P]$
is a toric algebra associated
to a finite, partially ordered set $P$ whose combinatorial properties
determine the geometry of the associated toric variety $\operatorname{Spec}(k[P])$.
They were introduced 1987 
by Takayuki Hibi in \cite{Hibi_rings}. Initially, Hibi rings were studied because of their 
appearance as flat degenerations of the
coordinate rings of Grassmannians and, more generally, flag varieties.
For a construction of these deformations consider \cite{Hodge}, and for a more modern construction using Sagbi bases consider \cite{Hibi_as_degen}. 

Nowadays Hibi rings are objects of interest in their own, and appear in various geometric,
algebraic and combinatorial contexts. Consider for example \cite{Alexdual}, \cite{Fatemeh}, \cite{Howe2005WeylCA}, \cite{Sangjib}.
The non-Gorenstein loci of Hibi rings have been subject to extensive study:
a Hibi ring
$k[P]$
is Gorenstein if and only if $P$ is a pure poset, that is, all maximal chains of $P$ have the
same length ([Hib87], Corollary 3.d).
In \cite{Measuring} it is shown that the non-Gorenstein locus is zero-dimensional if and only if each connected component of $P$ is pure.
In an attempt to find a good regularity condition weaker than the Gorenstein property, but stronger than the Cohen-Macaulay property, in \cite{Herzog}, Herzog, Hibi, and Stamate define $k[P]$ to be nearly Gorenstein if
$\operatorname{tr}(\omega)$ is the graded maximal ideal.
They prove that $k[P]$ is almost Gorenstein if and only if
all connected components $P_i$ of $ P$ are pure, and the difference in their ranks is bounded by one.
In 
\cite{Gor_loci_Janet} (Theorem 4.5), Miyazaki and Page give a description of the radical of $\operatorname{tr}(\omega)$
in combinatorial terms.
By proving Theorem \ref{theorem: max_comp} and Theorem \ref{theorem: graded_fibres}, we
give a new characterization of $\sqrt{\operatorname{tr}(\omega)}$ in a more conceptual, simple, and natural manner that unifies the above results.
Ultimately, we measure the deviation of $\operatorname{Spec}(k[P])$ from being Gorenstein, by how much $P$ deviates from being graded, in a precise sense, relating a geometric and a combinatorial notion of irregularity. These results take advantage of the discussion in section 2.

In Chapter five we proceed with an application of
Theorem \ref{thm: charac} to the non-toric case:
the first secant variety of the Segre variety
\[\mathbb{P}(V_1) \times \cdots  \times \mathbb{P}(V_n)  \xhookrightarrow{} \mathbb{P}(V_1 \otimes \cdots \otimes V_n) \]
is the closure of all tensors of rank $\leq 2$.
Secants of Segre varieties are objects in classical algebraic geometry that
are connected to the border rank of tensors, and to the computational complexity of matrix multiplication.
In \cite{Mateusz} the first secant variety, and in particular its singularities, are studied
using methods from statistics.
The most important tool are
secant-cumulant coordinates.
They are a special case of $L$-cumulants introduced in \cite{L-cumul}, and the induced coordinate change identifies affine, open patches
of the secant variety with certain toric varieties.
We extend the results of \cite{Mateusz} in the last chapter, the main result being Theorem \ref{thm: dim_Secants}.

\section{The toric trace ideal}

Our notation is close to \cite{cox2011toric}, which is the source that we refer to for information about toric varieties.

Let $N$ be a free abelian group of rank $n$ with dual $M = \operatorname{Hom}(N, \mathbb{Z})$,
and let
$\sigma \subseteq N_\mathbb{R} = N \otimes \mathbb{R}$ be a rational, pointed cone.
We denote its dual cone
$\sigma^\vee = \{l \in M_\mathbb{R}| \ l(\sigma) \subseteq \mathbb{R}_{\geq 0} \}$
and consider the normal toric ring $R = k[\sigma^\vee \cap M]$, where $k$ is an algebraically closed field in characteristic 0.

To a ray $\rho$ of $\sigma$ with primitive ray generator $u_\rho \in N$, we associate the torus-invariant prime-divisor
$D_\rho$. It is defined by the ideal
$I_ \rho$, spanned by all monomials away from the facet defined by $\rho$:
    \[
        I_\rho := \langle \chi^m | \  m \in \sigma^\vee \cap M, \  \langle u_\rho, m \rangle > 0 \rangle.
    \]
The negative sum $K = \sum_{ \rho  } - D_\rho$ of all torus-invariant, prime Weyl-divisors
is referred to as the canonical divisor of the toric variety $X = \operatorname{Spec}(R)$.
The module $\omega = \Gamma(\mathcal{O}_X(K)) $ is called the canonical module of $R$.
In \cite{Herzog} the non-Gorenstein locus of $X$ is described as the vanishing locus of its trace ideal:
\begin{definition}
The trace ideal of a module $M$ of $R$ is generated by the images of all $R$-module morphism $\varphi: M \longrightarrow R$.
\[
\operatorname{tr}(M) := \sum_{\varphi \in \operatorname{Hom}_R(M, R) } \varphi(M).
\]
\end{definition}

\begin{lemma} [\cite{Herzog}, Lemma 2.1]
\label{lemma:Herzog}
For every prime $\mathfrak{p} \in$ Spec$(R)$, $R_{\mathfrak{p}}$ is not Gorenstein if and only if    
$\operatorname{tr}(\omega) \subseteq \mathfrak{p}$.
\end{lemma}
In other words,
the non-Gorenstein locus of $X$ is the vanishing locus of
the radical $\sqrt{\operatorname{tr}(\omega)}$. 
In the remainder of this section we determine the minimal primes lying over $\operatorname{tr}(\omega)$, i.e. the irreducible components of the
non-Gorenstein locus.

The canonical module $\omega$ is spanned by all monomials
in the polyhedron 
\[
P_K = \{  x \in M_\mathbb{R} | \ \forall \rho   : \langle u_\rho, x \rangle \geq 1  \},
\]
which is obtained by translating all facets of $\sigma^\vee$ into the interior by one lattice-length.
Similarly, its dual $\omega^\vee= \operatorname{Hom}_R(\omega, R)$ is the module generated by the monomials in
the polyhedron
\[
P_{-K} = \{  x \in M_\mathbb{R} | \ \forall \rho  : \langle u_\rho, x \rangle \geq -1  \}.
\]
We now proceed with a polyhedral description of the homogeneous ideal tr($\omega$).
\begin{proposition}
\label{prop:gens}
The trace ideal $\operatorname{tr}(\omega)$ is generated by all monomials
$
\chi^{m + m'} , $ where $  m \in P_K \cap M,$ and $ m'\in P_{-K} \cap M$. In particular, tr$(\omega)$ is $M-$graded.
\end{proposition}
\begin{proof}
    By definition
    $\omega$ is the module of global sections of the sheaf $ \mathcal{O}(K) $
    associated to the canonical divisor.
    Its dual module  
    $\omega^\vee$ is the module of global sections $\Gamma(  \mathcal{O}(-K) )$:  
    \[
        \omega^\vee = \bigoplus_{ \operatorname{div}(\chi^m) - K \geq 0 } k \cdot \chi^m
        =  \bigoplus_{ m \in P_{-K} \cap M } k \cdot \chi^m.
    \] 
    Both $\omega$ and $\omega^\vee$ are submodules of the function field $K(X)$.
    Evaluation of an element of $\sigma^\vee$ at an element of $\sigma$ is multiplication within $K(X)$. So 
    $\operatorname{tr}(\omega)$ is the product $\omega^\vee \omega$ in $K(X)$,
    and
    the statement of the proposition follows.
\end{proof}

Let $\mathfrak{p}$ be a prime in $\operatorname{Spec}(R)$, minimal with the property that it contains $\operatorname{tr}(\omega)$.
Since $\operatorname{tr}(\omega)$ is graded, $\mathfrak{p}$ is as well, and by the orbit-cone correspondence
it is associated to a face $F$ of $\sigma^\vee$:
\[
\mathfrak{p} =  \langle \chi^m | \ m \in (\sigma^\vee \cap M) \setminus F    \rangle.
\]

To decide whether an arbitrary graded prime $\mathfrak{p}$, associated to a face $F$ of $\sigma^\vee$, contains $\operatorname{tr}(\omega)$, we make the following definition:
\begin{definition}
Let
    \[
    F[1] :=
    \{
        x \in M_\mathbb{R} | \ \forall \rho , \ u_\rho \in F^{\perp}:
         \langle u_\rho, x \rangle =  1
        \}
\] 
be the affine-linear space that is the intersection of all
facet-defining hyperplanes containing $F$, translated by one lattice-length into the interior of
$\sigma^\vee$.

\begin{example}
In this example $\sigma$ is the planar cone given by the inequalities
$ 2y \geq x, \ 2y \geq - x$. The face $F$ is the point $(0,0)$, and $F[1]$ is the point $(0, 0.5)$.

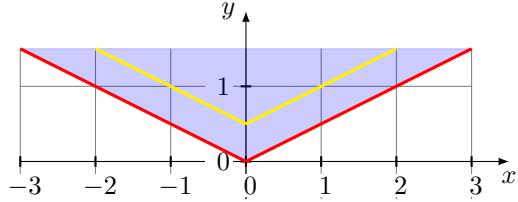
\begin{figure}[h]
\centering
     \begin{tikzpicture} 
     
        \tkzInit[xmax=3,ymax=1.5,xmin=-3,ymin=-0.5]
   \tkzGrid
   \tkzAxeXY

        \coordinate (label) at (-0.3 , 0.8);
        \coordinate (O) at (0,0);
        \coordinate (X1) at (3, 1.5) ;
        \coordinate (X2) at (-3,1.5) ;
        
                \coordinate (Ox) at (0,0.5);
        \coordinate (X1x) at (2, 1.5) ;
        \coordinate (X2x) at (-2,1.5) ;

        \filldraw [blue, opacity = 0.2] (O)--(X1)--(X2) ; 
                \filldraw [red, very thick] (O)--(X1) (O) --(X2) ; 
                \filldraw [yellow, very thick] (Ox)--(X1x) (Ox) --(X2x) ; 

    \end{tikzpicture}  
   
\end{figure}

\end{example}

\end{definition}
The following theorem characterizes the graded primes lying over tr$(\omega)$
in terms of $F[1]$:
\begin{theorem}
    \label{theorem: min_primes}
    Let $\mathfrak{p} \subseteq R$ be a graded prime defined by a face $F$ of $\sigma^\vee$. Then
    tr$(\omega) \subseteq \mathfrak{p}$ holds if and only if
    $F[1]$ does not contain a lattice point.
    \end{theorem}
\begin{proof}
Towards a contradiction we assume tr$(\omega) \subseteq \mathfrak{p}$,
and that there is a lattice point $z$ in $F[1]$.
Let $w$ be a lattice point in the relative interior of $F$.
Then
for every primitive generator $u_\rho$ of $\sigma$ the following inequalities hold.
\begin{align*}
    \langle u_\rho, w  \rangle > 0,  \text{ if } u_\rho \not \in F^\perp, \
    \langle u_\rho, w  \rangle = 0,  \text{ if } u_\rho \in F^\perp.
\end{align*}
After replacing $w$
with a positive integer multiple, we may assume the inequality
$ \langle u_\rho , w \rangle > \langle u_\rho, z\rangle + 1 $
to hold for all $u_\rho$ not in $F^\perp$. 
Then $w - z \in P_{-K} $
and $z +w \in P_K$.
We obtain
\[2w  =  (w  -z)   +  (z+w) \in P_{-K} \cap M + P_K \cap M. \]
By Proposition \ref{prop:gens}
$\chi^{2w}$ lies in $\sqrt{\operatorname{tr}(\omega)}$.
$\chi^{2w}$ is not contained in $\mathfrak{p}$, and hence the radical of tr($\omega$)
is not contained in $\mathfrak{p}$. Thus, tr($\omega$) is not contained either.\\ 
 
For the other direction let 
$\chi^w \in   \operatorname{tr}(\omega) \setminus \mathfrak{p}$.
Then $w$ lies in $F$, and by
Proposition \ref{prop:gens}
there is an element $z \in P_K \cap M$
with $w - z \in P_{-K}$:
\[
       -\langle u_\rho ,z \rangle= \langle u_\rho ,w - z \rangle \geq -1
\]
holds for all $u_\rho$ in $F^\perp$.
So $\langle u_\rho, z \rangle \leq 1$, and hence $\langle u_\rho ,z \rangle = 1$ since $z \in P_K$.
We obtain
$z \in F[1]$.
\end{proof}

\begin{remark}
Theorem \ref{theorem: min_primes} can easily
be generalized to normal toric varieties that are
not affine. We use the notation from \cite{cox2011toric}.
Let $X$ be a normal toric variety defined by a fan $\Sigma$ of rational,
pointed cones in $N_\mathbb{R}$,
$Y \subseteq X$ a non-empty torus-invariant subvariety, given by a
cone $\sigma \in \Sigma$.
\end{remark}
\begin{theorem}
    \label{thm: charac}
    The subvariety $Y$ is contained in the non-Gorenstein locus $Z$ if and only if
    there is no element $m$ of $M$, such that
    $\langle u_\rho, m \rangle = 1$
    holds for every ray $\rho$ of $\sigma$.
\end{theorem}
\begin{proof}
    Both $Z$ and $Y$ are closed.
    Since $Y$ intersects the affine open scheme $\operatorname{Spec}( k[ \sigma^\vee \cap M] )$ in 
    $\operatorname{Spec}( k[ \sigma^\perp \cap M] )$,
    we may replace $Y$ with $\operatorname{Spec}( k[ \sigma^\perp \cap M] )$,
    $X$ with $\operatorname{Spec}( k[ \sigma^\vee \cap M] )$, and $Z$ with the non-Gorenstein locus of $\operatorname{Spec}( k[ \sigma^\vee \cap M] )$.

    By Lemma \ref{lemma:Herzog} and Theorem \ref{theorem: min_primes},
     $Y$ is contained in $Z$ if and only if for the choice
    $F = \sigma^\perp$, $F[1]$ does not contain a lattice point. In other words,
    if there is no element $m$ of $M$ with
    \[
   u_\rho \in (\sigma^\perp)^\perp \implies \langle u_\rho, m \rangle =  1  \]
    for every ray $\rho$ of $\sigma$.
    All ray generators of $\sigma$ are contained in 
$
        (\sigma^\perp)^\perp = \sigma - \sigma.
$
\end{proof}

\section{Non-Gorenstein loci of Hibi rings}

Hibi rings are certain toric rings associated to finite, partially ordered sets.
We recall a description of the the associated polyhedral cone, and apply
the results from the previous chapter.
Finally, Theorem \ref{thm: charac} describes the non-Gorenstein locus as a combinatorial
measure for how far posets deviate from being graded, relating a geometric and a combinatorial notion of irregularity.

    We call a set $P$, together with a transitive, reflexive order $\leq$,
    a partially ordered set, or poset.
    To elements
    $a \leq b$ of $P$ we associate the interval
    $
    [a,b] : =   \{ x \in P | \ a \leq x \leq b\},
    $
    and call $P$ bounded if 
    $P$ is the interval
$
       [a, b]
$
for some elements $a$ and $b$.

    For different elements $a \leq b$,
    the covering relation $a \lessdot b$ is defined to hold if
$
    \#[a, b]     = 2.
$
That is, no elements lie properly between $a$ and $b$.

    We call a totally ordered poset
   $a_1 \leq \cdots \leq a_r$
   of cardinality $r$
    a chain of length $r-1$.
    $P$ is defined to be pure if all chains contained in $P$,
    maximal with respect to inclusion,
    have the same length.

    The set $\mathcal{I}(P)$ of subsets $I \subseteq P$, that are closed from below,
    is called the lattice of order ideals.

\begin{definition}
    Let $P$ be a finite poset and let $k[t, x_p, \ p \in P]$
    be the free $k$-algebra in the variables
    $x_p$ for all $p$ in $P$, and the variable $t$.
    For each order ideal $I$ we denote the monomial
    \[
    x^I := \prod_{p \in I} x_p.    
    \]
    The $k$-algebra
$R(P)$
    generated by all monomials $tx^I$, where $I$ runs over all poset-ideals,
    is the Hibi-ring associated to $P$.
\end{definition}
In fact,
    $R(P)$ is the normal toric ring
$
    k[C(P) \cap M]    
$
    associated to a cone $C(P)$, and of Krull-dimension $\#P + 1$
(\cite{Hibi_rings}).
The cone
    \[
        C(P)  = \{  \psi: \overline{P} \longrightarrow \mathbb{R} | \ \forall a, b \in \overline{P}:
        a \leq b \implies \psi(a) \geq \psi(b) \ , \psi(\infty) = 0   \}
    \] 
   consists of the order-reversing maps from 
$
    \overline{P} := P \dot{\cup} \{-\infty, \infty\}    
$
    to $\mathbb{R}$, taking $0$ as minimal value.
It is the cone over the order-polytope 
    \[
    Q(P) :=    \{ \psi \in C(P)|\ \psi(- \infty) = 1 \}  .
    \]
The generators of $R(P)$, the order ideals, and the vertices of $Q(P)$,
are in natural bijection, as can be seen by associating to an order ideal $I$ the vertex $\psi_I$ of 
the order- polytope $Q(P)$:
\begin{align*}
\psi_I :  \overline{P}  &\longrightarrow \mathbb{Z} \\
  p & \longmapsto 
            \begin{cases}
                1, & \text{for } p \in I \cup \{ -\infty\} \\
                0, & \text{for } p \notin I \cup \{- \infty\} \\
            \end{cases}.
\end{align*}

The faces of $C(P)$ have a combinatorial description.
The maximal proper faces, called facets, are in bijective correspondence to
the order relations:
\[
F_{a \lessdot b  } :=   \{ \psi \in C(P)  | \ \psi(a) = \psi(b)   \}.
\]
More generally, faces of $C(P)$ are families of functions $\psi$ that are constant along certain equivalence relations on $P$:
\begin{definition}
Let $P$ be a finite poset. A quotient poset of $P$ is a poset $P'$ together with a surjective, order preserving map
$\phi : P \longrightarrow P'$ with connected fibres, such that the order relation on $P'$ is the transitive hull of
the relation $a' \leq' b' := \exists a \in \phi^{-1}(a'),\ b \in \phi^{-1}(b') \mid \ a \leq b$.
\end{definition}

\begin{definition}
    \label{def: Fphi}
    Let $\phi: \overline{P} \twoheadrightarrow P'$ be a quotient poset of $\overline{P}$.
We denote by $F_\phi$ the face
\[F_\phi := \underset{\substack{ a \lessdot b, \\ \phi(a) = \phi(b) }}{\bigcap}   F_{a \lessdot b} \]
of $C(P)$, consisting of the functions $\psi$ that are constant on the fibres of $\phi$.
\end{definition}
This definition bijectively identifies quotient posets and (possibly empty) faces of $C(P)$ (\cite{Geissinger}).
The inclusion order of faces corresponds to the refinement order of those equivalence relations $\sim_\phi$, that identify the fibres of $\phi$.

\begin{example}
    \label{exmp: poset}
    Consider the partially-ordered set 
$
    P =  \{ p_1, p_2, p_3 \}    
$
    with only the relation $p_1 \geq p_2$.

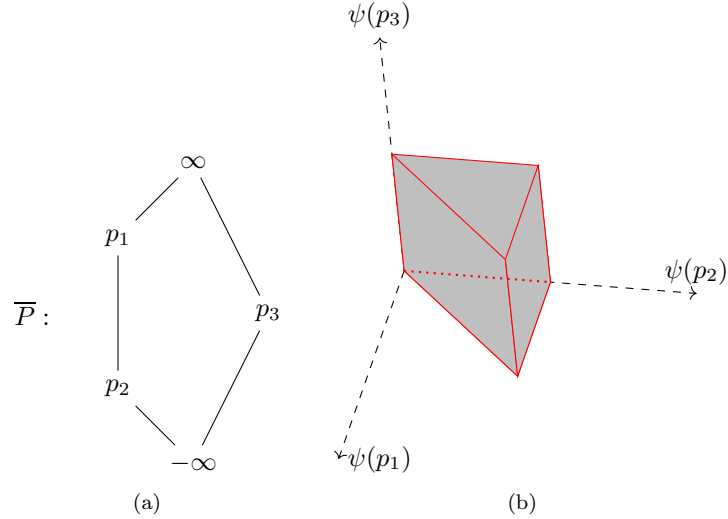
\begin{figure}[H]
\centering
 \subfloat[][]{
   \begin{tikzpicture}  
            \filldraw[black] (4.5,1)node[anchor=west] {$\overline{P}:$};

            [scale=.9,auto=center,every node/.style={circle,fill=red!10}] 
              
            \node (a2) at (6,0) {$p_2$};  
            \node (a1) at (6,2)  {$p_1$}; 
            \node (a3) at (8,1)  {$p_3$};  
            \node (a4) at (7,3) {$\infty$};
            \node (a5) at (7,-1) {$-\infty$};

            \draw (a1) -- (a2); 
            \draw (a1) -- (a4);
            \draw (a3) -- (a4);
            \draw (a2) -- (a5);
            \draw (a3) -- (a5);
          \end{tikzpicture}  
 }
 \subfloat[][]{

\tdplotsetmaincoords{60}{130}

    \begin{tikzpicture}[fill=lightgray, tdplot_main_coords, scale = 2]

        \tdplotsetrotatedcoords{12}{10}{15}

        \begin{scope}[draw=red, tdplot_rotated_coords, axis/.style={->,dashed}]

            \coordinate (A1) at (0,0,0);
            \coordinate (B1) at (0,0,1) ;
            \coordinate (A2) at (0,1,0) ;
            \coordinate (B2) at (0,1,1) ;
            \coordinate (A3) at (1, 1, 0) ;
            \coordinate (B3) at (1,1,1);

            \draw[axis, black] (0, 0, 0) -- (2, 0, 0) node [right] {$\psi(p_1)$};
            \draw[axis, black] (0, 0, 0) -- (0, 2, 0) node [above] {$\psi(p_2)$};
            \draw[axis, black] (0, 0, 0) -- (0, 0, 2) node [above] {$\psi(p_3)$};

        \draw[fill] (A1) -- (A2) -- (A3) -- cycle;
        \draw[fill] (A1) -- (A2) -- (B2) -- (B1) -- cycle;
        \draw[fill] (A1) -- (A3) -- (B3) -- (B1) -- cycle;
        \draw[fill] (A2) -- (A3) -- (B3) -- (B2) -- cycle;

        \draw[dotted, thick] (A1) -- (A2);

        \end{scope}


        \end{tikzpicture}          
 }
\end{figure}



There are nine faces of $Q(P)$ of dimension one, corresponding to order preserving, surjective maps
to the unique bounded, partially ordered set $  q_0 \leq q_1 \leq q_2  $ with three elements.
We give a list
that matches the one-dimensional faces $F$ to the fibres of $\phi$.

\begin{center}
\begin{tabular}{ |c|c|c|c|c|c } 
\hline
Nr. & $F$ & $\phi^{-1}(q_2) $ & $\phi^{-1}(q_1) $ & $\phi^{-1}(q_0) $ \\
\hline
1   & conv($ \{ (0, 0, 0), \ (0,0, 1)  \} $)    & $\{ \infty, p_1, p_2 \}$  & $\{p_3\}$ & $\{-\infty\}$\\ 
2   & conv($ \{ (0, 0, 0), \ (0,1, 0)  \} $)    & $\{ \infty, p_1, p_3 \}$  & $ \{p_2\} $ & $\{-\infty\}$\\ 
3   & conv($ \{ (0, 0, 0), \ (1, 1, 0)  \} $)   & $ \{ \infty, p_3 \}$ & $\{p_1, p_2\}$& $\{-\infty\}$\\ 
4   & conv($ \{ (0, 0, 1), \ (0, 1, 1)  \} $)   & $\{ p_1, \infty  \} $ & $\{p_2\}$ & $\{ p_3, -\infty \}$ \\ 
5   & conv($ \{ (0, 0, 1), \ (1, 1, 1)  \} $)   & $\{\infty\}$ & $\{ p_1, p_2  \}$& $\{ p_3, - \infty\}$ \\ 
6   & conv($ \{ (0, 1, 0), \ (0, 1, 1)  \} $)   & $ \{ p_1,  \infty\}$ & $\{p_3\}$ & $\{ p_2, -\infty \}$\\ 
7   & conv($ \{ (0, 1, 0), \ (1, 1, 0)  \} $)   & $ \{ p_3,  \infty\}$ & $\{p_1\}$& $\{ p_2, -\infty \}$\\ 
8   & conv($ \{ (0, 1, 1), \ (1, 1, 1)  \} $)   & $\{\infty\}$ & $\{p_1\}$& $\{p_3, p_2, -\infty \}$ \\ 
9   & conv($ \{ (1, 1, 0), \ (1, 1, 1)  \} $)   & $\{\infty\}$ & $\{p_3\}$ & $\{ p_1, p_2, -\infty \}$ \\ 
\hline
\end{tabular}
\end{center}

\end{example}

Applying the results from Chapter 3 now allows us to characterize the non-Gorenstein locus
in terms of non-graded subsets of $P$:
\begin{definition}
    Let $P$ be a finite poset. We call $P$ graded if
    there is an order-reversing map $ \psi: P \longrightarrow \mathbb{Z}$,
    such that for every covering relation $a \lessdot b $ it holds $ \psi(a) = \psi(b) +1 $. 
    $\psi$ is called a grading of $P$.
\end{definition}
Let the face $F_\phi$ of $C(P)$ be defined by a quotient poset $\phi: \overline{P}  \twoheadrightarrow P'$, as in
definition \ref{def: Fphi}.
\begin{theorem}
    \label{theorem: graded_fibres}
    $F_\phi[1]$ contains a lattice point if and only if every fibre of $ \phi $, equipped with the restricted order relation,  is graded.
\end{theorem}
\begin{proof}
    By definition, $F_\phi$ is the intersection
    \[
    F_\phi = \bigcap_{ \substack{ a \lessdot b, \\ \phi(a) = \phi(b) }  }  F_{a \lessdot b}
    \]
    of all facets $F_{a \lessdot b}$, where $a$ and $b$ are elements in the same fibre.
    Each such facet is supporting, so we obtain
    \[
        F_\phi[1] = 
     \{  \psi \in M_\mathbb{R} | \ \forall a \lessdot b, \  \phi(a) = \phi(b) : \ 
    \psi(a) = \psi(b) + 1   \},
    \]
     and consequently the lattice points of $F_\phi[1]$
    are functions $\psi$
    that define a grading of each fibre of $\phi$, showing the implication from left to right.
	 Conversely, a separate choice of gradings on each fibre together form an integral element of $F[1]$.
\end{proof}

The remainder of this chapter is devoted to characterizing the maximal components of the non-Gorenstein locus.

\begin{lemma}
The subsets $A \subseteq \overline{P}$
that appear as fibres
of quotient poset maps $\phi: \overline{P} \longrightarrow P'$
are determined by the property that they be connected and
\[a \leq b \leq c, \ a, c \in A \implies b \in A. \]
\end{lemma}
\begin{proof}
To see that every such subset $A$ appears as a fibre, define
$P' := P \setminus A \cup\{  \star \}$ with the natural quotient map $\phi_A: \overline{P} \rightarrow P'$, and equip it with
the transitive closure of the relation
$a' \leq' b' := \exists a \in \phi_A^{-1}(a'),\ b \in \phi_A^{-1}(b') \mid \ a \leq b$.
The other direction is clear.
\end{proof}

We call subsets $A \subseteq \overline{P}$, appearing as fibres, complete.
For any quotient poset morphism $\phi$, having a non-graded fibre $A$,
consider the face
$F_{\phi_A}$ defined in the proof above.
It determines a subvariety of the non-Gorenstein locus
that is a superset of the variety determined by
 $F_\phi$.
We obtain:
\begin{theorem}
\label{theorem: max_comp}
The map 
\[
A \longmapsto
F_{\phi_A}
\]
bijectively identifies
the maximal components of the non-Gorenstein locus with subsets $A$ of $\overline{P}$, that are minimal with the property that 
they be complete and not graded.
\end{theorem}

\begin{corollary}
    \label{cor: dim_locus}
    The dimension of the non-Gorenstein locus is
$ \max  \{ \#P - \#A + 2 \},
$
where $A$ runs over all non-graded, complete subsets of $\overline{P}$.
\end{corollary}
\begin{proof}
The cone $C(P)$ has dimension 
$\#P + 1$ and the cone $C(P \setminus A \cup\{  \star \})$
is of dimension $\#P - \#A + 2$.
\end{proof}

\begin{corollary}
    The codimension of the non-Gorenstein locus is at least 4.
\end{corollary}
\begin{proof}
	Every non-graded sub-poset $A \subseteq \overline{P}$ has at least 5 elements.
\end{proof}

\begin{remark}
    Corollary \ref{cor: dim_locus}, generalizes the already known characterizations of
    Gorenstein Hibi rings and Hibi rings that are Gorenstein on the pointed spectrum.
    By the corollary, the non-Gorenstein locus is empty if and only if $\overline{P}$ is graded.
    Since $\overline{P}$ is a bounded poset, it is graded if and only if it
    is pure, which holds if and only if $P$ is pure.
    
    Similarly, by the corollary, the locus is zero-dimensional if and only if $P \dot{\cup} \{ \infty \}$
    and $P \dot{\cup} \{- \infty \}$ are graded.
    It is easy to show that this is equivalent to every connected component of $P$ being pure
    (for a proof consider Lemma 5.2 in \cite{Herzog}).
\end{remark}

\section{Comparison to \cite{Gor_loci_Janet}}
In the paper \cite{Gor_loci_Janet},
non-Gorenstein loci of Hibi rings are studied.
In particular, a family of graded ideals is
described in Theorem 4.5, that intersect in the radical ideal
$\sqrt{\operatorname{tr}(\omega)}$.
In this chapter we compare results and deduce Theorem 4.5 from the discussion in Chapter 3.\\

To state Theorem 4.5 we need the definitions of rank and distance of elements $a \leq b$ in a poset $P$:
the rank rank$(a, b)$ is defined to be the maximal
length $r-1$ of a chain $a = a_1 \lessdot a_2 < \cdots \lessdot a_r = b.$
Similarly, the distance dist$(a, b)$ is defined to be the minimal length $r-1$ of an inclusion-maximal chain
$a = a_1 \lessdot a_2 \lessdot \cdots \lessdot a_r = b. $
\begin{definition}
\label{def: a_i_b_j_vs_p}
Let $u$ be a natural number and
\[
a_1 < b_1 > a_2 < \cdots > a_u < b_u > a_1
\]
be elements of $\overline{P}$,
satisfying the inequality
\begin{equation}
\label{eq: inequality}
\sum_{i = 1}^u \operatorname{rank}(a_i, b_i) > \operatorname{dist}(a_2, b_1) + \cdots + \operatorname{dist}(a_u,b_{u-1}) +\operatorname{dist}(a_1,b_{u})  .
\end{equation}
We define the graded prime ideal
\[
\mathfrak{p}_{(a_1, \dots, a_u, b_1, \dots, b_u)} := \langle \chi^\psi| \psi \in C(P), \ \psi \text{ nonconstant on } \{a_1, \dots, a_u, b_1, \dots, b_u \}  \rangle.
\]
\end{definition}

\begin{theorem}{(4.5, \cite{Gor_loci_Janet})}
The ideal $\sqrt{\operatorname{tr}(\omega)}$ is the intersecion of primes
\[
\sqrt{\operatorname{tr}(\omega)} = \bigcap_{(a_1, \dots, a_u, b_1, \dots, b_u)} \mathfrak{p}_{(a_1, \dots, a_u, b_1, \dots, b_u)}.
\]
\end{theorem}
 
By our description of the minimal primes lying over $\operatorname{tr}(\omega)$,
the monomials $\chi^\psi$ that lie in the radical $\sqrt{\operatorname{tr}(\omega)}$ are characterized by the property 
that $\psi$ be non-constant on all non-graded, complete subsets of $\overline{P}$.
The observation to make is that, 
given a  complete subset $A$ of $\overline{P}$,
there exist elements $(a_1, \dots, a_u, b_1, \dots, b_u) $ in $A$ satisfying inequality \eqref{eq: inequality}, if and only if $A$ is not graded.
This proves the inclusion from right to left.
Conversely,
the union of all intervals $[a_i, b_j]$ forms a non-graded, complete subset $A$ of $\overline{P}$, showing the other inclusion.

\section{Secants of Segre varieties} 

We start by introducing notation: let $k_1 \leq \dots \leq k_n$ be natural numbers, $n \geq 2$, and for each index $i$ let
$V_i$ denote a linear space of dimension $k_i$.
The image of the Segre map
\begin{align*}
\mathbb{P}(V_1) \times \cdots  \times \mathbb{P}(V_n)  & \longrightarrow \mathbb{P}(V_1 \otimes \cdots \otimes V_n)   \\
(v_1, \dots, v_n) & \longmapsto v_1 \otimes \cdots \otimes v_n
\end{align*}
consists of all rank one tensors.
Its first secant variety $\operatorname{Sec}(k_1, \dots, k_n)$
is the Zariski-closure of the set of all tensors of
rank $\leq 2$.
In \cite{Mateusz}, an affine open covering of
$\operatorname{Sec}(k_1, \dots, k_n)$ is constructed, such that
each patch is isomorphic to the trivial vector bundle of rank $k_1 + \cdots + k_n$ over a certain toric variety $X$.
$X$ is the spectrum of the toric $k$-algebra
$k[\sigma^\vee \cap \mathbb{Z}^{1 + k_1 + \cdots + k_n} ]$, where $\sigma^\vee$ is the polyhedral cone
\begin{align*}
    \sigma^\vee =  
    \{
        &\left(q_0,  q^1_1, \dots, q^1_{k_1}, \dots, q^n_1, \dots, q^n_{k_n}    \right) \in \mathbb{R}^{1 + (k_1 + \cdots + k_n)}
        \mid \\
        &q^i_j \geq 0 \ \forall 1 \leq i \leq n, 1\leq j \leq k_i,\\
        & q_0 - \sum_{j= 1}^{k_i} q^i_j \geq 0 \ \text{for } 1 \leq i \leq n,\\
        &\sum_{i  = 1}^n \sum_{j= 1}^{k_i}  q^i_j -2q_0 \geq 0 \}.
\end{align*}
$\sigma^\vee$ is the cone over the polytope $Q = C \cap \{ q_0 = 1 \}$, which is the product of $n$ simplices $ \Delta_{k_i }$
of respective dimension $k_i$, intersected with the halfspace
$\{   \sum_{i,j} q^i_j \geq 2 \} $.

Using Theorem \ref{theorem: min_primes}, one can show that the non-Gorenstein locus of the secant variety
is the trivial vector bundle of rank $k_1 + \cdots + k_n$ over the locus of 
$X$. This is done by expressing the faces of the cone $\sigma^\vee  \times \mathbb{R}_{\geq 0}$, that contribute to the non-Gorenstein locus, as products
$F \times \mathbb{R}_{\geq 0}$.

We from now on assume that the non-Gorenstein locus is not empty.
According to \cite{Mateusz} (Theorem 7.18), this happens in all but the following cases.
\begin{itemize}
\item $n =5,: \ k_5 =1$,
\item $n =3: \ (k_1, k_2, k_3) \in \{ (1,1,1), (1,1,3), (1,3,3), (3,3,3)   \},$
\item $n =2: \ k_2 = k_1$, or $k_1 = 1$.
\end{itemize}

In the the remainder of this section we investigate which positive-dimensional faces $F$ of $\sigma^\vee$ contribute to non-Gorenstein locus
of $X$, by applying Theorem \ref{theorem: min_primes}.
We show that any maximal such face is a cone over a product
$
 \Delta_{k_l } \times  \Delta_{k_m } 
$
for distinct indices $l$ and $m$, and decide the existence of such faces.
Ultimately we obtain
\begin{theorem}
    \label{thm: dim_Secants}

    If $X$ is not Gorenstein,
    the non-Gorenstein locus of $X$ is of dimension
    \begin{align*}
    &\bullet \emph{max} 
	    \{
	    k_{l} +k_{m} +1,  \text{ where } l \neq m , \  \sum _{i \neq l, m} k_i \neq 3 \}
	 & \text{if } n \geq 4 \text{ or } n = 3, k_1 > 1,\\
     & \bullet k_{2} +k_{3} +1 & \text{ if } n = 3,  k_2 \neq 1 ,\\
     & \bullet 0 & \text{ else}.
    \end{align*}
\end{theorem}
\begin{proof}
By construction of the cone $\sigma^\vee$, its dual cone $\sigma$ is generated by the vectors
\begin{itemize}
    \item $R^i_j := e^i_j, \ \ \forall 1 \leq i \leq n, 1\leq j \leq k_i$,
    \item $L_i := e_0 - \sum_{j= 1}^{k_i} e^i_j , \ \ \forall 1 \leq i \leq n$,
    \item $S:= \sum_{i  = 1}^n \sum_{j= 1}^{k_i}  e^i_j -2e_0 $.
\end{itemize}
Here the vectors $e^i_j$ denote the standard basis of $\mathbb{R}^{k_1 + \cdots + k_n}$.
A list of the primitive ray generators of $\sigma$ is given in the proof of Theorem
7.18 in \cite{Mateusz}:
in the case $n=4$, all generators of $\sigma$ are ray generators.
If $n = 3$, all $R^i_1$ with $k_i = 1$ are omitted.

We now proceed with a proof of the theorem for the case
$n \geq 3$,  and omit the case $n =2$. It is analogous, but easier, except
that for $n = 2$ the lineality space of $\sigma$ is spanned by $S, L_1, L_2$,
and the ray generators of 
$\sigma / \langle S, L_1, L_2 \rangle$ are the vectors $[ R^i_j \operatorname ]$, so we need to replace
$\sigma$ with $\sigma / \langle S, L_1, L_2 \rangle$, and
$\sigma^\vee$ with $ \sigma^\vee \cap  \langle S, L_1, L_2 \rangle^\perp$.\\

In the case $n=3$, $\sigma$ is a pointed cone,
and
we may apply Theorem \ref{theorem: min_primes}.
Let $F \subseteq \sigma^\vee$ be a positive-dimensional face.
An integral element of $F[1]$ is an intregral solution $(q^0,q^i_j)  \in \mathbb{R}^{1 +k_1 + \cdots + k_n}$ to all affine-linear equations
\begin{itemize}
\item $ q^i_j = 1  $, if $F \subseteq (R^i_j)^\perp $,
\item $ q_0 - \sum_{j= 1}^{k_i} q^i_j  = 1  $, if $F \subseteq L_i^\perp $,
\item $ \sum_{i  = 1}^n \sum_{j= 1}^{k_i}  q^i_j -2q_0 = 1 $, if $F \subseteq S^\perp $.
\end{itemize}
As can be seen by direct computation, such a solution always exists if $F \subseteq S^\perp $ does not hold, or if only for one ray generator $L_i$ it holds $F \subseteq L_i$.
Furthermore, $F$ is of dimension zero if $F \subseteq S^\perp $, and in addition there are at least three inclusions of the form $F \subseteq L_i$.

Let from now on $F \subseteq S^\perp $ and let $l$ and $m$ be the only two distinct indices such that $L_l$ and $L_m$  satisfy $F \subseteq L_l$, $F \subseteq L_m$.
We now investigate when $F[1]$ does not contain an integral element.
Observe that for each vector $R^i_j$ with $i \neq l, m$ it holds $F \subseteq (R^i_j)^\perp $.
If every such vector is a ray generator, by adding the three
three equations
\begin{gather*}
q_0 - \sum_{j= 1}^{k_{l}} q^{l}_j = 1, \  q_0 - \sum_{j= 1}^{k_{m}} q^{m}_j = 1, \
\sum_{i  = 1}^n \sum_{j= 1}^{k_i}  q^i_j -2q_0 = 1,
\end{gather*}
we obtain
\begin{equation}
\label{eq: 3 =}
3 = \sum_{i  \neq l, m } \sum_{j= 1}^{k_i}  q^i_j  = \sum_{i  \neq l, m } k_i.
\end{equation}
Consequently, $F[1]$ does not contain an integral element if equation \eqref{eq: 3 =} fails.
In the case $n = 3$ this is only true if $k_i \neq 1$ holds for every $i \neq l, m$,  since our argument uses that all $R^i_j$, $i \neq l, m$ be ray generators. 

If on the other hand equation \eqref{eq: 3 =} holds, 
the choice
$q_0 =k_{l} +k_{m} $ and
$ q^i_j = 1  $ for all $i$ and $j$ except $q^{l}_1 = k_{m}, \ q^{m}_1 = k_{l}$
defines an integral element of $F[1]$. 

The maximal faces $F$ contained in $S^\perp \cap  L_l^\perp \cap L_m^\perp$ are of the form $F = C \cap L_l^\perp \cap L_m^\perp \cap S^\perp$, a cone over the product of simplices $F \cap Q =  \Delta_{k_l } \times  \Delta_{k_m } $, and of dimension $k_l + k_m +1$.
This proves the desired statement.

\end{proof}

\bibliographystyle{alpha}
\bibliography{bibfile}



\end{document}